\documentclass[12pt, a4paper, reqno]{amsart}
\usepackage{tikz}
\usepackage{url}
\usetikzlibrary{arrows.meta}
\usepackage{amsmath,amsthm,amsfonts,amssymb,latexsym,mathrsfs,color,extarrows}
\usepackage{color, graphicx, comment}
\allowdisplaybreaks

\newtheorem{Theorem}{Theorem}%[section]
\newtheorem{Corollary}[Theorem]{Corollary}
\newtheorem{Proposition}[Theorem]{Proposition}

\DeclareMathOperator{\des}{des}
\DeclareMathOperator{\inv}{inv}
\DeclareMathOperator{\ssum}{sum}
\DeclareMathOperator{\dist}{dist}
\DeclareMathOperator{\maj}{maj}
\DeclareMathOperator{\noz}{noz}
\DeclareMathOperator{\tel}{tel}
\DeclareMathOperator{\uel}{uel}

\newcommand{\tf}[2]{f_{#1}^{(#2)}}
\newcommand{\tg}[2]{g_{#1}^{(#2)}}

\usepackage{graphicx} % 必须加载
\newcommand{\qbinom}[2]{\left[ \genfrac{}{}{0pt}{}{#1}{#2} \right]_q}

\title{The inversion number statistic for inversion sequences}

\author{Lora R. Du}
\address{Center for Applied Mathematics and KL-AAGDM,
Tianjin University,
Tianjin 300072, P.R. China
}
\email{loradu@tju.edu.cn}
\author{Guo-Niu HAN}
\address{I.R.M.A., UMR 7501, Universit\'e de Strasbourg
et CNRS, 7 rue Ren\'e Descartes, F-67084 Strasbourg, France}
\email{guoniu.han@unistra.fr}

\subjclass[2020]{05A05, 05A10, 05A15, 05A19, 05A30}

\keywords{Inversion sequences, inversions, permutations, Catalan numbers, involutions, $q$-derivative operator}

\date{April 10, 2026}

\begin{document}

\begin{abstract}
Inversion sequences, also known as subexcedant sequences, form a fundamental class of objects in enumerative combinatorics.
In this paper, we study the joint distribution of five statistics on inversion sequences.
While several statistics on inversion sequences have been extensively investigated, our contribution is to introduce the inversion number statistic, originally defined for permutations, into the context of inversion sequences. As special cases, we recover classical permutation statistics, including the Stirling, Mahonian and Eulerian distributions, as well as the Catalan and Narayana numbers.  Somewhat unexpectedly, our specializations also include the number of involutions in the symmetric group.

Our study arises from a $q$-analog of Comtet's expansion formula obtained by substituting the classical derivative operator $D$ with 
the $q$-derivative operator $D_q$.
\end{abstract}

\maketitle

\section{Introduction}
Inversion sequences, also known as subexcedant sequences, form a fundamental class of objects in enumerative combinatorics \cite{Baril-Vajnovszki-2017, Corteel-Martinez-Savage-2016, Huh-Kim-Seo-Shin-2026, Kim-Lin-2017, Mantaci-Rakotondrajao-2001, Martinez-Savage-2018, Savage-Schuster-2012}.
An {\it inversion sequence} of length $n$ is a sequence $e = (e_0, e_1, \ldots, e_{n-1})$ of nonnegative integers such that $0 \le e_i \leq  i$ for each~$i$. The set of all inversion sequences of length $n$ is denoted by $I_n$. An inversion sequence can be represented  geometrically, as illustrated in Fig.~\ref{exam-statis} for the inversion sequence $e=(0,0,0,2,4,0,5)$. 
\begin{figure}[htbp]
    \centering
    \begin{tikzpicture}[
    scale=0.6,  square/.style={draw=black, minimum size=0.6cm, inner sep=0}]
    \foreach \y/\n in {0/7, 1/6, 2/5, 3/4, 4/3, 5/2, 6/1} {
        \foreach \x in {1,...,\n} {  
            \node[square] at (8-\x, \y) {};
        }
        \node at (8, \y) {\small \y};
    }
    \foreach \pos in {(7,5), (6,0), (5,4), (4,2), (2,0), (1,0), (3,0)} {
        \node  at \pos {$\bullet$};
    }
    \end{tikzpicture}
    \caption{A geometric representation of the inversion sequence.}
    \label{exam-statis}
\end{figure}
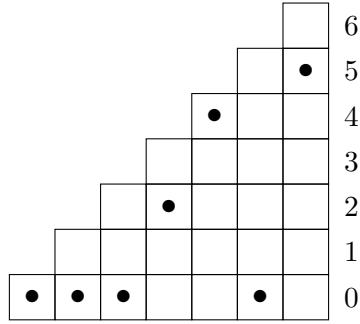
For an inversion sequence $e=(e_0, e_1, \ldots, e_{n-1})\in I_n$, we define the following five statistics: 
\begin{align*}
\inv(e) &= \bigl|\{(i,j) : i<j,\ e_i>e_j\}\bigr|,\\[5pt]
\ssum(e) &= \sum_{i=0}^{n-1} e_i,\\[5pt]
\noz(e) &= \bigl|\{j : e_j=0\}\bigr|,\\[5pt]
\tel(e) &= n - {\rm dist}(e),\\[5pt]
\uel(e) &= n - \max(e)-1,
\end{align*}
where $\dist(e)$ is the number of distinct entries of $e$.

These statistics are, respectively, the number of inversions, the sum of its entries,  the number of zero entries, the total number of empty lines (i.e.,  integers in  $\{0,1,\ldots,n-1\}$ not present in $e$), the number of upper empty lines (i.e., integers  in  $\{0,1,\ldots,n-1\}$ not present in $e$ but greater than $\max(e)$).
For example, the inversion sequence  $e=(0,0,0,2,4,0,5)\in I_7$ has $\inv(e)=2$, ${\rm sum}(e)=11$, $\noz(e)=4$, $\tel(e)=3$, and  $\uel(e)=1$.

While several statistics on inversion sequences have been extensively investigated (see, for example, \cite{Baril-Vajnovszki-2017, Huh-Kim-Seo-Shin-2026, Kim-Lin-2017, Martinez-Savage-2018, Savage-Schuster-2012}), our main contribution is the extension of the inversion number statistic, originally defined for permutations, into the context of inversion sequences. 
To study the joint distribution of these statistics, we define the five-variable polynomial
$$
F_n(x;y,z,p,q)=\sum_{e\in I_n} x^{\noz(e)}y^{\tel(e)}z^{\uel(e)}p^{\ssum(e)}q^{\inv(e)},\quad (n\ge 1),
$$
also denoted by $F_n(x)$ when no ambiguity arises.
The first few values of $F_n(x)$ are listed below:
\begin{align*}
F_1(x) &= x,\\[5pt]
F_2(x) &= x^2yz + xp,\\[5pt]
F_3(x) &= x^3 y^2z^2  + x^2 (pqyz + p^2y+pyz) + x(p^2yz +  p^3).
\end{align*}

We derive a recurrence relation for  $F_n(x)$ using the $q$-derivative operator $D_q$.
For a polynomial $f(x)$ in $x$, the $q$-{\it derivative operator} \cite[p.~22]{GR90} is defined by
\begin{equation}\label{q-operator}
    D_qf(x):=\frac{f(x)-f(xq)}{ x - xq}.
\end{equation}
Associated with this operator is the $q$-exponential operator
$$
T_q = \sum_{k \geq 0} \frac{(D_q)^k}{[k]_q!},
$$
where the $q$-factorial is given by $[0]_q!=1$, and for $k\geq 1$,
$$
[k]_q! = [k]_q [k-1]_q \cdots [1]_q,
$$
with the $q$-integer
$
[k]_q = 1+q+q^2 + \cdots + q^{k-1}
$.

\begin{Theorem} \label{Main-F-rec}
The polynomials $\{F_n(x)\}_{n\geq 1}$ satisfy  $F_1(x) = x$, and for $n\geq 1$,
\begin{equation}\label{F-rec}
F_{n+1}(x) = (z-1)  z^{n-1} y^n  x^{n+1}  
	+ p^nx \Bigl(
	(y-1)F_n\Big(\frac{x}{p}\Big) + T_q(F_n(x))\Big|_{x=\frac{x}{p}}\Bigr).
\end{equation}
\end{Theorem}

The proof of Theorem \ref{Main-F-rec} is presented in Section 2. In Section 3, we demonstrate that $F_n(x)$ serves as a master polynomial unifying several classical combinatorial numbers through  appropriate  structures. When $q=1$, the underlying combinatorial structure corresponds to  the symmetric group. In this case, each choice of parameters yields a different specialization: setting the variables $(x;y,z,p,q)$ to $(x;1,1,1,1)$, $(1;y,1,1,1)$, $(1;1,z,1,1)$ and $(1;1,1,p,1)$ recovers, respectively, generating functions for the Stirling  numbers of the first kind, Mahonian statistics, the distribution of the maximum inversion entry, and the Eulerian numbers on permutations. 
%In this case, the polynomial $F_n(x; y,z,p,1)$ specializes to a polynomial in $x,  y, z, p$ that  refines the Stirling, Mahonian,  the maximum inversion table entry and Eulerian statistics on permutations. 
The case  $q=0$ specializes the model to  Dyck paths. Specifically, setting $x=y=z=p=1$ gives the Catalan number $C_n = F_n(1;1,1,1,0)$. Moreover, the specializations $(x;1,1,1,0)$, $(1;y,1,1,0)$ and $(x;1,x,1,0)$ yield generating functions whose coefficients correspond to the number of returns to ground level, the Narayana numbers, and the sum of the heights of the first and last peaks, respectively. Somewhat unexpectedly, when $q=-1$ and $x=y=z=p=1$, we recover the number of involutions in the symmetric group. 
In Section 4, we derive the generating polynomial for inversions over inversion sequences with a fixed number of occurrences of each entry $0,1,\dots,n-1$.  Moreover, a special family of inversion sequences, in which each nonzero entry appears exactly once, is shown to provide a combinatorial interpretation for the $q$-Stirling numbers of the second kind.  

For $n\geq 1$, let $f_n(q)$ be the generating polynomial for the inversion number statistic of inversion sequences, i.e.,
$
f_n(q)=F_n(1;1,1,1,q)=\sum_{e\in I_n} q^{\inv(e)}.
$
We reproduce the first few values of $f_n(q)$, as well as the specializations when $q=1,0,-1$:
{
\renewcommand{\arraystretch}{1.2}
$$
\begin{array}{| c | c |   c  | c |   c   |   }
\hline
n & f_n(q) & f_n(1) & f_n(0) & f_n(-1)     \\
\hline
 1 &         1  &         1  &    1  &     1 \\
 2 &         2  &         2  &    2  &     2 \\
 3 &       q+5   &         6  &    5  &     4 \\
 4 &       3q^2 + 7q+14   &      24  &    14  &     10 \\
 5 &       3q^4 + 11q^3 + 28q^2 + 36q + 42  &       120  &   42  &    26 \\
\hline
\end{array}
$$
}

\medskip

Our combinatorial framework arises from the construction of a $q$-analog of Comtet's expansion formula. In 1823, Scherk \cite[Appendix A]{Blasiak2011F} derived the following expansion formula  for  $(xD)^n $ :
\begin{equation}\label{eq:Scherk}
(xD)^n  = \sum_{k=0}^n S(n,k) x^k D^k,
\end{equation}
where $S(n,k)$ are the Stirling numbers of the second kind.
Subsequently, in 1973, Comtet \cite{Comtet-1973, Han-Ma-2024} investigated the more general operator $(g(x)D)^n$ and obtained an explicit  formula.  Let $f_k = D^k(f)$  and $g_k = D^k (g)$. Comtet's result can be  written as  
\begin{equation}\label{Ank-def}
(gD)^nf =\sum_{k=1}^nL_{n,k}{f}_k,
\end{equation}
where $L_{n,k}$ depends only on $g$ and its derivatives.  Furthermore, for $1\leqslant k\leqslant n$, the coefficients $L_{n,k}$ are given by 
\begin{equation}\label{Ank-Comtet}
L_{n,k}=\frac{g}{k!}\sum (2-k_1)(3-k_1-k_2)\cdots (n-k_1-\cdots-k_{n-1})\frac{g_{k_1}}{k_1!}\cdots\frac{g_{k_{n-1}}}{k_{n-1}!},
\end{equation}
where the sum is taken over all sequences $(k_1,k_2,\ldots,k_{n-1})$ of nonnegative integers such that
$k_1+k_2+\cdots+k_{n-1}=n-k$ and $k_1+\cdots+k_j\leqslant j$ for any $1\leqslant j\leqslant n-1$.

In Section 5, we show that the expression $(gD_q)^nf$  admits a combinatorial interpretation in terms of inversion sequences. Using the generating polynomial for inversion sequences with fixed multiplicities obtained in Section 4, we arrive naturally at the $q$-analog of Comtet's formula, see Theorem  \ref{th1:q-Comtet}.  Moreover, specializing to $g(x)=x$ yields a $q$-analog of Scherk's formula,  see Theorem \ref{th:q-Scherk}.

\section{Proof of Theorem \ref{Main-F-rec}}

We now describe a construction from $I_n$ to $I_{n+1}$.  Given a sequence $e'=(e_0',\ldots,e_{n-1}')\in I_n$, define $e=(e_0,\ldots,e_n)\in I_{n+1}$ by setting $e_{0} = 0$, and for $0\leq i\leq n-1$, 
$$
e_{i+1} = 
\begin{cases}
  e_{i}'+1, & \text{if } e_{i}'\neq 0,\\[5pt]
  0 \text{ or 1}, &\text{if } e_i' = 0.
\end{cases}
$$
 Each zero entry in $e'$ independently yields a binary choice (0 or 1). Hence, an element $e'\in I_n$ with exactly $k$ zeros generates  $2^k$ distinct sequences in $I_{n+1}$.  It is straightforward to verify that
 this construction produces all elements of $I_{n+1}$. 

Before we track how the statistics encoded by  $x,y,z,p$ and $q$ behave under this construction, we first  apply  the $q$-exponential operator $T_q$ on the monomial $x^n$, which naturally encodes the weighted binary choices arising from zero entries in our construction:
$$
T_q (x^n) = 
\sum_{k=0}^n \frac{[n]_q!}{[n-k]_q!} 
\frac{x^{n-k}}{[k]_q!}
=
\sum_{k=0}^n  \qbinom{n}{k}x^k,
$$
where $\qbinom{n}{k}$ denotes the $q$-binomial coefficient.  
In particular, $\qbinom{n}{k}$ counts binary sequences of length $n$ with $k$ ones and $n-k$ zeros, weighted by the inversion statistic (see \cite{Andrews-1976}). 

Next, we examine how the statistics encoded by  $x,y,z,p$ and $q$ are transformed under this construction.   

 \noindent \textbf{Case 1: The all-zero sequence.} 
    Let $e'=(0,\ldots,0)\in I_n$, with weight 
    \[A(x):=x^ny^{n-1}z^{n-1}.\]
     
    \noindent (1a) All zeros in $e'$ remain $0$, illustrated in Fig.~\ref{Case1a}. Then the number of zeros, the total number of  empty lines,  and the number of upper empty lines each increase by $1$, while the number of inversions and the sum of entries remain unchanged. The contribution is 
   $x^{n+1}y^nz^n = xyzA(x).$ 
    
   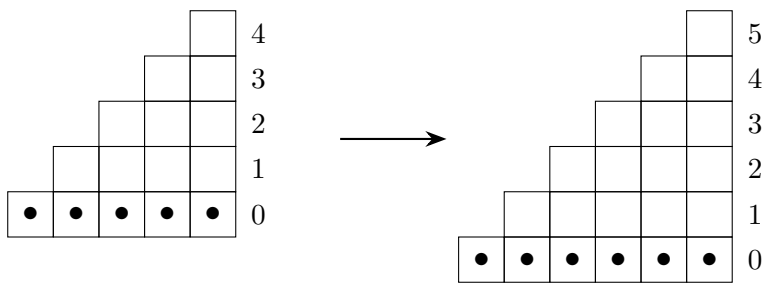
\begin{figure}[htbp]
    \centering
    \begin{tikzpicture}[baseline=0pt,
    scale=0.6,  square/.style={draw=black, minimum size=0.6cm, inner sep=0}]
    \foreach \y/\n in { 0/5, 1/4, 2/3, 3/2, 4/1} {
        \foreach \x in {1,...,\n} {  
            \node[square] at (6-\x, \y) {};
        }
       \node at (6, \y) {\small \y};
    }
    \foreach \pos in {  (5,0), (4,0), (2,0), (1,0), (3,0)} {
        \node  at \pos {$\bullet$};
    }
    \end{tikzpicture}\qquad %  
    \begin{tikzpicture}[baseline=-1cm]  
       \draw[-{Stealth[scale=1.2]}, thick] (0,0) -- (1.4,0); 
    \end{tikzpicture}  
    \begin{tikzpicture}[baseline=0.6cm ,
    scale=0.6,  square/.style={draw=black, minimum size=0.6cm, inner sep=0}]
    \foreach \y/\n in { 0/6, 1/5, 2/4, 3/3, 4/2,5/1} {
        \foreach \x in {1,...,\n} {  
            \node[square] at (7-\x, \y) {};
        }
        \node at (7, \y) {\small \y};
    }
    \foreach \pos in { (6,0), (5,0), (4,0), (2,0), (1,0), (3,0)} {
        \node  at \pos {$\bullet$};
    }
    %  \foreach \pos in { (6,1), (5,1), (4,1), (2,1) , (3,1)} {
    %     \node  at \pos {$\textcolor{red}{\circ}$};
    % }

    % \draw[dashed, ->](6,0.8) -- (6,0.2); \draw[dashed, ->](5,0.8) -- (5,0.2);
    %  \draw[dashed, ->](4,0.8) -- (4,0.2); \draw[dashed, ->](3,0.8) -- (3,0.2); \draw[dashed, ->] (2,0.8) -- (2,0.2);
    \end{tikzpicture}
    \caption{The mapping $e' \to e$ in Case 1a.}
    \label{Case1a}
\end{figure}

\medskip 
\noindent (1b) At least one zero in $e'$ is mapped to $1$, shown in Fig.~\ref{Case1b}. Then the total number of empty lines and the number of upper empty lines  remain invariant. Suppose $j$ $(0\leq j\leq n-1)$ zeros remain $0$ and  $n-j$ zeros become $1$. The sum of entries increases  by $n-j$ and the $q$-binomial coefficient~$\qbinom{n}{j}$ accounts for the inversions. Summing over  $0\leq j\leq  n-1$,  the contribution is 
\begin{align*}
xy^{n-1}z^{n-1}\sum_{j=0}^{n-1}\qbinom{n}{j}x^j p^{n-j} &= xy^{n-1}z^{n-1} p^n   \sum_{j=0}^{n} \qbinom{n}{j} \left(\frac{x}{p}\right)^j -  xA(x)  \\[5pt]
&= xp^n T_q(A(x)) \Big|_{x = \frac{x}{p}} - xA(x).
\end{align*}

  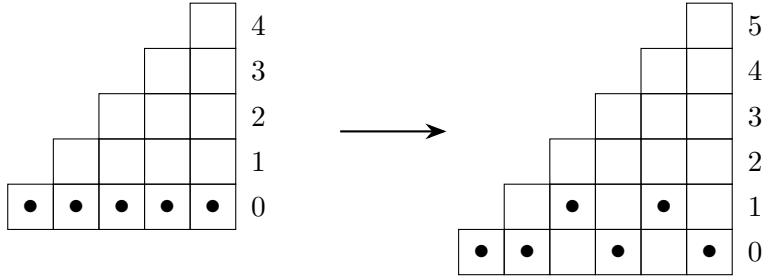
\begin{figure}[htbp]
    \centering
    \begin{tikzpicture}[baseline=0pt,
    scale=0.6,  square/.style={draw=black, minimum size=0.6cm, inner sep=0}]
    \foreach \y/\n in { 0/5, 1/4, 2/3, 3/2, 4/1} {
        \foreach \x in {1,...,\n} {  
            \node[square] at (6-\x, \y) {};
        }
        \node at (6, \y) {\small \y};
    }
    \foreach \pos in {  (5,0), (4,0), (2,0), (1,0), (3,0)} {
        \node  at \pos {$\bullet$};
    }
    \end{tikzpicture}\qquad %  
    \begin{tikzpicture}[baseline=-1cm]  
       \draw[-{Stealth[scale=1.2]}, thick] (0,0) -- (1.4,0); 
    \end{tikzpicture}  
    \begin{tikzpicture}[baseline=0.6cm ,
    scale=0.6,  square/.style={draw=black, minimum size=0.6cm, inner sep=0}]
    \foreach \y/\n in { 0/6, 1/5, 2/4, 3/3, 4/2,5/1} {
        \foreach \x in {1,...,\n} {  
            \node[square] at (7-\x, \y) {};
        }
        \node at (7, \y) {\small \y};
    }
    \foreach \pos in { (6,0), (5,1), (4,0), (2,0), (1,0), (3,1)} {
        \node  at \pos {$\bullet$};
    }
    \end{tikzpicture}
    \caption{The mapping $e' \to e$ in Case 1b.}
    \label{Case1b}
\end{figure}

\noindent \textbf{Case 2: General sequences with $1\leq k<n$ zeros.}
Let $ I_{n}^{(k)}$ be the subset of $I_n$ with exactly $k$ zeros. For $e'\in I_{n}^{(k)}$, each of the $n-k$ non-zero entries is  incremented by $1$, contributing a factor of $p^{n-k}$.  

    \smallskip 
  \noindent (2a) All $k$ zeros in $e'$ are mapped to $0$, see Fig.~\ref{Case2a}. Both the number of  zeros and the  total number of empty lines increase by 1, contributing a factor of $xy$. The number of upper empty lines and inversions remain constant.  This subcase contributes: 
   \begin{align*}
     xy\sum_{k=1}^{n-1} x^{k}p^{n-k} \sum_{e\in I_{n}^{(k)}}y^{\tel(e)}z^{\uel(e)}p^{\ssum(e)}q^{\inv(e)} =   xyp^{n}F_n\Big(\frac{x}{p}\Big)-xyA(x).
   \end{align*}
  
 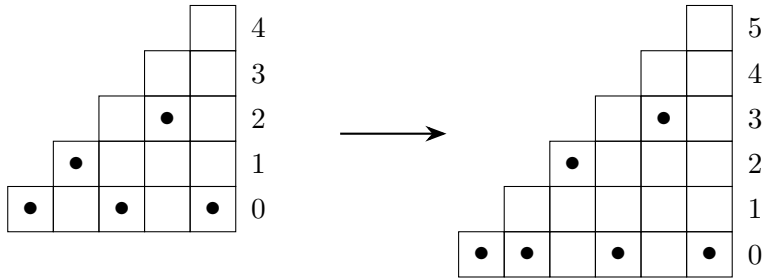
\begin{figure}[htbp]
    \centering
    \begin{tikzpicture}[baseline=0pt,
    scale=0.6,  square/.style={draw=black, minimum size=0.6cm, inner sep=0}]
    \foreach \y/\n in { 0/5, 1/4, 2/3, 3/2, 4/1} {
        \foreach \x in {1,...,\n} {  
            \node[square] at (6-\x, \y) {};
        }
        \node at (6, \y) {\small \y};
    }
    \foreach \pos in {  (5,0), (4,2), (2,1), (1,0), (3,0)} {
        \node  at \pos {$\bullet$};
    }
    \end{tikzpicture}\qquad %  
    \begin{tikzpicture}[baseline=-1cm]  
       \draw[-{Stealth[scale=1.2]}, thick] (0,0) -- (1.4,0); 
    \end{tikzpicture}  
    \begin{tikzpicture}[baseline=0.6cm ,
    scale=0.6,  square/.style={draw=black, minimum size=0.6cm, inner sep=0}]
    \foreach \y/\n in { 0/6, 1/5, 2/4, 3/3, 4/2,5/1} {
        \foreach \x in {1,...,\n} {  
            \node[square] at (7-\x, \y) {};
        }
        \node at (7, \y) {\small \y};
    }
    \foreach \pos in { (6,0), (5,3), (4,0), (2,0), (1,0), (3,2)} {
        \node  at \pos {$\bullet$};
    }
    %  \foreach \pos in { (6,1), (5,1), (4,1), (2,1) , (3,1)} {
    %     \node  at \pos {$\textcolor{red}{\circ}$};
    % }

    % \draw[dashed, ->](6,0.8) -- (6,0.2); \draw[dashed, ->](5,0.8) -- (5,0.2);
    %  \draw[dashed, ->](4,0.8) -- (4,0.2); \draw[dashed, ->](3,0.8) -- (3,0.2); \draw[dashed, ->] (2,0.8) -- (2,0.2);
    \end{tikzpicture}
    \caption{The mapping $e' \to e$ in Case 2a.}
    \label{Case2a}
\end{figure}

\noindent (2b) At least one zero is mapped to $1$, depicted in Fig.~\ref{Case2b}. Assume that~$i$ of the zeros stay as $0$ while the remaining $k-i$ zeros are changed to~$1$, where $0 \le i \le k-1$.
Then the total number of empty lines and the number of upper empty lines remain unchanged. The inversion contribution 
 is~$\qbinom{k}{i}$. In this case, the contribution is: 
\begin{align*}
 &   x\sum_{k=1}^{n-1}p^{n-k}\sum_{i=0}^{k} \qbinom{k}{i} x^ip^{k-i}  \sum_{e\in I_{n}^{(k)}}y^{\tel(e)}z^{\uel(e)}p^{\ssum(e)}q^{\inv(e)} \\[5pt]
 &-  \Big(xp^nF_n\Big(\frac{x}{p}\Big)-xA(x)\Big) \\[5pt]
% &=   xp^{n}\sum_{k=1}^{n}\sum_{i=0}^{k} \qbinom{k}{i} (\frac{x}{p})^i  \sum_{e\in I_{n}^{(k)}}y^{\tel(e)}z^{\uel(e)}p^{\ssum(e)} q^{\inv(e)}-xp^{n} T_q(A(x))|_{x=\frac{x}{p}} -  (xp^nF_n\bigl(\frac{x}{p}\bigr)-xA(x)) \\[5pt]
&=   xp^{n}T_q(F_n(x))\Big|_{x=\frac{x}{p}}-xp^{n} T_q(A(x))\Big|_{x=\frac{x}{p}} - \Big(xp^nF_n\Big(\frac{x}{p}\Big)-xA(x)\Big).
\end{align*}

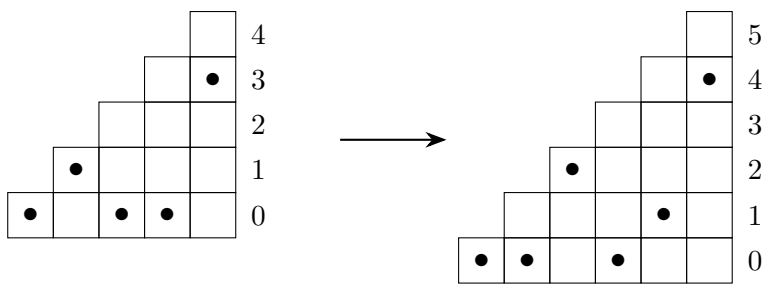
\begin{figure}[htbp]
    \centering
    \begin{tikzpicture}[baseline=0pt,
    scale=0.6,  square/.style={draw=black, minimum size=0.6cm, inner sep=0}]
    \foreach \y/\n in { 0/5, 1/4, 2/3, 3/2, 4/1} {
        \foreach \x in {1,...,\n} {  
            \node[square] at (6-\x, \y) {};
        }
        \node at (6, \y) {\small \y};
    }
    \foreach \pos in {  (5,3), (4,0), (2,1), (1,0), (3,0)} {
        \node  at \pos {$\bullet$};
    }
    \end{tikzpicture}\qquad %  
    \begin{tikzpicture}[baseline=-1cm]  
       \draw[-{Stealth[scale=1.2]}, thick] (0,0) -- (1.4,0); 
    \end{tikzpicture}  
    \begin{tikzpicture}[baseline=0.6cm ,
    scale=0.6,  square/.style={draw=black, minimum size=0.6cm, inner sep=0}]
    \foreach \y/\n in { 0/6, 1/5, 2/4, 3/3, 4/2,5/1} {
        \foreach \x in {1,...,\n} {  
            \node[square] at (7-\x, \y) {};
        }
        \node at (7, \y) {\small \y};
    }
    \foreach \pos in { (6,4), (5,1), (4,0), (2,0), (1,0), (3,2)} {
        \node  at \pos {$\bullet$};
    }
    %  \foreach \pos in { (6,1), (5,1), (4,1), (2,1) , (3,1)} {
    %     \node  at \pos {$\textcolor{red}{\circ}$};
    % }

    % \draw[dashed, ->](6,0.8) -- (6,0.2); \draw[dashed, ->](5,0.8) -- (5,0.2);
    %  \draw[dashed, ->](4,0.8) -- (4,0.2); \draw[dashed, ->](3,0.8) -- (3,0.2); \draw[dashed, ->] (2,0.8) -- (2,0.2);
    \end{tikzpicture}
    \caption{The mapping $e' \to e$ in Case 2b.}
    \label{Case2b}
\end{figure}

Collecting all contributions from each case above leads to  the   recurrence relation for  $F_{n+1}(x)$:
\begin{align*}
    F_{n+1}(x)
    &=(z-1)z^{n-1}y^{n}x^{n+1} +  p^nx \Bigl(
	(y-1)F_n\Big(\frac{x}{p}\Big) + T_q(F_n(x))\Big|_{x=\frac{x}{p}}\Bigr).
\end{align*}
This completes the proof.

\section{Applications}

In this section, we explore several specializations of the polynomial $F_n(x; y, z, p, q)$, highlighting its role as a master polynomial that unifies diverse combinatorial structures. By setting $q=1,0,-1$, the polynomial recovers classical combinatorial frameworks: when $q=1$, it reflects the structure of the symmetric group; when $q=0$, it encodes Dyck paths; and when $q=-1$, it captures involutions.

\subsection{Special Cases for $q=1$}

When $q=1$, the $q$-exponential operator $T_q$ becomes the shift operator $ f(x) \mapsto f(x+1)$. Furthermore, setting $y=z=1$ allows the generating function $F_{n}(x;1,1,p,1)$  to be expressed as an explicit product formula. This provides a direct connection to Mahonian statistics and Stirling statistics. 

\begin{Proposition} 
For $n\geq 1$, 
\begin{equation}\label{eq:yzq=1}
F_n(x;1,1,p,1) = (p^{n-1}+ \cdots  + p^2 + p + x)\cdots (p^2 + p + x)(p + x)x.
\end{equation}
\end{Proposition}
\begin{proof}
When $q=y=z=1$, the  recurrence relation of $F_n(x)$ given in~\eqref{F-rec} simplifies to 
$$
F_{n+1}(x) =  p^nx  F_n\Big(\frac{x}{p}+1 \Big).
$$

Starting from $F_1 = x$, an iterative application of this recurrence immediately yields \eqref{eq:yzq=1}.
\end{proof}

Consequently, we have the following corollaries of the product form:
\begin{enumerate}
    \item  Setting $p=1$ gives $F_n(x;1,1,1,1)=x(x+1)\cdots(x+n-1)$, which is the generating function for the Stirling statistics on $\mathfrak{S}_n$. In particular, the coefficient of $x^k$ is the unsigned Stirling numbers of the first kind $c(n,k)$.
     \item Setting $x=1$ yields $F_n(1;1,1,p,1)=[n]_p!$. This is the generating function for the  Mahonian distribution on the symmetric group $\mathfrak{S}_n$, such as the major index $\maj(\sigma)$, defined as the sum of indices $i$ such that $\sigma_i>\sigma_{i+1}$.  
\end{enumerate}

   In the special case where $x=y=p=q=1$, the recurrence \eqref{F-rec} simplifies significantly, allowing us to determine the distribution of $z$ explicitly. 

\begin{Proposition}
 For $n\geq 1$,  we have 
    \[F_{n}(1;1,z,1,1) = \sum_{j=0}^{n-1}  M(n,j) z^j,\]
    where the coefficients are given by
    \begin{equation}\label{eq:xypq=1}
       M(n,j) = (n-j)! \; [   ({n-j}+1)^{j+1} -      ({n-j})^{j+1} ] .
    \end{equation}
\end{Proposition}

\begin{proof}
 Setting $y=p=q=1$ in \eqref{F-rec} yields 
$$
	F_{n+1}(x) = (z-1)  z^{n-1}  x^{n+1}  + x F_n(x+1).
$$
Iterating the above recurrence, we obtain
\begin{align*}
	F_{n+1}(x) 
	&= \sum_{j=1}^n  x (x+1) \cdots (x+j-2) (x+j-1)^{n-j+2} (z-1)  z^{n-j}   \\
	&~~ +  x (x+1) \cdots (x+n). 
\end{align*}
To complete the proof, we evaluate $F_{n+1}(1)$ and extract the coefficient of  $z^j$.  
\end{proof}

\noindent \textbf{Remark}. The sequence $M(n,n-j-1)$ matches OEIS A056151 \cite{OEIS}. Specifically, $M(n,n-j-1)$ counts the number of permutations $\sigma=\sigma_1\sigma_2\cdots \sigma_n\in \mathfrak{S}_n$ such that $\max(\sigma_i - i) = n-j-1$.  

Finally, we consider the total number of empty lines statistic $\tel(e)$. It is an established result that the distribution of this statistic over $I_n$ coincides with the classical descent statistic on permutations. For a permutation $\sigma = \sigma_1\sigma_2 \cdots \sigma_n \in \mathfrak{S}_n$, we define its descent set as $\text{Des}(\sigma) = \{i \in [n-1] : \sigma_i > \sigma_{i+1}\}$. The number of descents  of $\sigma$  is  $\text{des}(\sigma) = |\text{Des}(\sigma)|$.

\begin{Proposition}[\cite{Mantaci-Rakotondrajao-2001}]
For $n\geq 1$, setting $x=z=p=q=1$, we have 
    $$
    F_n(1;y,1,1,1) = \sum_{e\in I_n} y^{\tel(e)} = \sum_{\sigma\in \mathfrak{S}_n} y^{\des(\sigma)}.
    $$ 
\end{Proposition}

\subsection{Special Cases for $q=0$}

When $q=0$, inversion sequences reduce to weakly increasing sequences, that is,  for any $e\in I_n$, we have $$e_0\leq e_1\leq e_2\leq \cdots \leq e_{n-1}.$$
Let $WI_n$ denote the set of    weakly increasing sequences of length $n$. 

\begin{Proposition}[\cite{Martinez-Savage-2018}]
    The number of weakly increasing inversion sequences of length $n$ is the $n$-th Catalan number 
    $$C_n=\frac{1}{n+1}\binom{2n}{n}.$$
\end{Proposition}

\begin{proof}
Let $\mathcal{D}_n$ denote the set of lattice paths from $(0,0)$ to $(n,n)$  consisting of $n$ east steps $E=(1,0)$ and $n$ north steps $N=(0,1)$ that never go above the diagonal $y=x$. It is well known that  $|\mathcal{D}_n|=C_n$. We establish a  bijection $\phi$ between the set $WI_n$  and the set $\mathcal{D}_n$.
 
Given $e=(e_0,e_1,\dots,e_{n-1}) \in WI_n$, define a lattice path $\phi(e)\in \mathcal{D}_n$ by 
 $$
\phi(e)=EN^{e_1-e_0}EN^{e_2-e_1} \cdots EN^{e_{n-1}-e_{n-2}}EN^{n-e_{n-1}},
$$
where $N^k$ denotes $k$ consecutive north steps. 
The total number of north steps is
\[(e_1-e_0)+(e_2-e_1)+\cdots+(e_{n-1}-e_{n-2})+(n-e_{n-1})=n.\]
So $\phi(e)$ is a path from $(0,0)$ to $(n,n)$. 

For example, taking $n=4$ and $e=(0,1,1,2)$, the resulting path is $ENEENEN^2$, as illustrated in Fig.~\ref{exam-0112}. 
 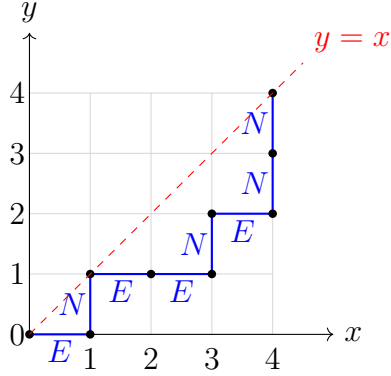
\begin{figure}[htbp]
 \centering
\begin{tikzpicture}[scale=0.8]
    \draw[gray!30, step=1] (0,0) grid (4,4);
    \draw[->] (0,0) -- (5,0) node[right] {$x$};
    \draw[->] (0,0) -- (0,5) node[above] {$y$};
    \foreach \x in {1,2,3,4} \draw (\x, -0.4) node {$\x$};
    \foreach \y in {0,1,2,3,4} \draw (-0.2, \y) node {$\y$};
    \draw[thick, blue] (0,0) -- (1,0) node[midway, below, yshift=2pt] {$E$}
    -- (1,1) node[midway, left, xshift=3pt] {$N$} -- (2,1) node[midway, below, yshift=2pt] {$E$}
    -- (3,1) node[midway, below, yshift=2pt] {$E$}
    -- (3,2) node[midway, left, xshift=3pt] {$N$}
    -- (4,2) node[midway, below, yshift=2pt] {$E$}
    -- (4,3) node[midway, left, xshift=3pt] {$N$}
    -- (4,4) node[midway, left, xshift=3pt] {$N$};
    \fill (0,0) circle (2pt);
    \fill (4,4) circle (2pt);
    \fill (1,0) circle (2pt) ;
    \fill (1,1) circle (2pt) ;
    \fill (2,1) circle (2pt) ;
    \fill (3,1) circle (2pt);
    \fill (3,2) circle (2pt) ;
    \fill (4,2) circle (2pt) ;
    \fill (4,3) circle (2pt) ;
    \draw[dashed, red] (0,0) -- (4.5,4.5) node[above right] {$y=x$};
\end{tikzpicture}
\caption{The lattice path $\phi(e)\in \mathcal{D}_4$.}
\label{exam-0112}
\end{figure}

To verify that $\phi(e)$ never rises above $y = x$, we  observe that after $e_i-e_{i-1}$ north steps, the path reaches $(i,e_i)$.  Since  $e_{i} \leq i$ for each $i$,  the path stays  below the diagonal. Hence, $\phi(e)\in\mathcal{D}_n$.

Conversely, given a path $P \in \mathcal{D}_n$, define  $e = (e_0, e_1, \dots, e_{n-1})$ by letting $e_i$ be the $y$-coordinate of $P$ immediately before the $(i+1)$-st east step. Then $e_{i+1} - e_i$ record the number of consecutive north steps between two east steps,  obviously, $e_{i+1}-e_i\geq 0$.  So $e$ is weakly increasing.  Moreover,  the condition that $P$ never goes above $y = x$ implies $e_i \leq  i$ for each $i$. So $e\in WI_n$ is  a weakly increasing inversion sequence. 

It is straightforward to verify that the two constructions are inverse to each other, completing the proof.
\end{proof}

 While the specialization $q = 0$ identifies the cardinality of $WI_n$ with the $n$-th Catalan number, the refined polynomials reveal deeper structural symmetries and combinatorial properties, particularly in terms of Dyck paths. Recall that a Dyck path of semilength $n$ consists of up steps $(1,1)$ and down steps $(1,-1)$ from $(0,0)$ to $(2n,0)$, never going below the $y$-axis. In this context, we adopt the following definitions: a {\it  peak} is an up step immediately followed by a down step, a {\it valley} is a down step immediately followed by an up step, and a {\it return} is a down step ending at ground level ($y = 0$); see \cite{Deutsch-1999, Stanley-2015}.

\begin{Proposition}
Let $p=y=1$ and $f_n(x,z)=zF_n(x;1,z,1,0)$. Then  for $n\geq 1$,  
\begin{enumerate}
    \item $f_n(x,z)$ is symmetric in  $x$ and $z$, that is, 
$ f_n(x,z) = f_n(z,x)$.
\item When $x=z$, we have 
$$f_{n}(x,x) = \sum_{k=2}^{2n} H(n,k) x^{k},$$
where $H(n,k)$ is the number of  Dyck paths of semilength $n$  such that the sum of  the heights of the first peak and the last peak is $k$ (OEIS A114503). 
\end{enumerate}
\end{Proposition}

\begin{proof}
	Let  $e=(e_0,e_1,\ldots,e_{n-1})\in WI_n$. Then \( e_{n-1} \) is the maximum entry of \( e \). As established earlier, there exists a bijection \( \phi \) between the set $WI_n$ and the set of lattice paths \( \mathcal{D}_n \). Under the bijection $\phi$ between $WI_n$ and $\mathcal{D}_n$, the following correspondences hold: 

\noindent (a). The number of zeros in  \( e \) corresponds to the number of east steps on the line \( y = 0 \);
 
 \noindent (b). The statistic $\uel(e)+1$  equals the number of final  consecutive north steps, namely \( n - e_{n-1} \).
 
 Here the shift by $1$ arises naturally from the path interpretation, and ensures that this statistic ranges over $[1,n]$, matching the range of the number of zeros.
    
   To prove the symmetry \( f_n(x,z) = f_n(z,x) \), consider the  involution on $\mathcal{D}_n$ obtained by reversing the path and exchanging $E$ and $N$. This operation preserves the set $\mathcal{D}_n$ and interchanges the two statistics above. Consequently, the generating function $f_n(x,z)$ is symmetric in $x$ and~$z$.
   For example, Fig.~\ref{exam-0013} is the involution of the lattice path illustrated in Fig.~\ref{exam-0112}.
    \begin{figure}[htbp]
 \centering
\begin{tikzpicture}[scale=0.75]
    \draw[gray!30, step=1] (0,0) grid (4,4);
    \draw[->] (0,0) -- (5,0) node[right] {$x$};
    \draw[->] (0,0) -- (0,5) node[above] {$y$};
    \foreach \x in {1,2,3,4} \draw (\x, -0.4) node {$\x$};
    \foreach \y in {0,1,2,3,4} \draw (-0.2, \y) node {$\y$};
    \draw[thick, blue] (0,0) -- (1,0) node[midway, below, yshift=2pt] {$E$}
    -- (2,0) node[midway, below, yshift=2pt] {$E$} -- (2,1) node[midway, left, xshift=3pt] {$N$}
    -- (3,1) node[midway, below, yshift=2pt] {$E$}
    -- (3,2) node[midway, left, xshift=3pt] {$N$}
    -- (3,3) node[midway, left, xshift=3pt] {$N$}
    -- (4,3) node[midway, below, yshift=2pt] {$E$}
    -- (4,4) node[midway, left, xshift=3pt] {$N$};
    \fill (0,0) circle (2pt);
    \fill (4,4) circle (2pt);
    \fill (1,0) circle (2pt) ;
    \fill (2,0) circle (2pt) ;
    \fill (2,1) circle (2pt) ;
    \fill (3,1) circle (2pt);
    \fill (3,2) circle (2pt) ;
    \fill (3,3) circle (2pt) ;
    \fill (4,3) circle (2pt) ;
    \draw[dashed, red] (0,0) -- (4.5,4.5) node[above right] {$y=x$};
\end{tikzpicture}
\caption{An involution for lattice path $ENEENENN$.}
\label{exam-0013}
\end{figure}
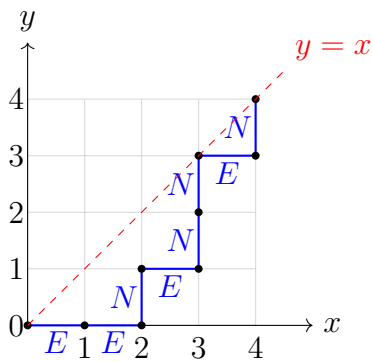

For the second statement, map each path in $\mathcal{D}_n$ to a Dyck path of semilength $n$ by replacing $E=(1,0)$ with $(1,1)$ and $N=(0,1)$ with $(1,-1)$. Under this transformation:
\begin{itemize}
\item $x$ records the height of the first peak;
\item $z$ records the height of the last peak.
\end{itemize}
Setting \( z=x \),  the exponent records the sum of the heights of the first and the last peak.  Therefore,
$$xF_n(x;1,x,1,0) = \sum_{k=2}^{2n} H(n,k) x^{k},$$
as claimed.
\end{proof}

The next result describes the distribution of valleys via the statistic $\tel(e)$. 
\begin{Proposition} 
 For $n\geq 1$, setting  $q=0$ and $x=z=p=1$, we have 
\[F_n(1;y,1,1,0) = \sum_{k=0}^{n-1} N(n,k) y^k, \]
where $N(n,k)=\frac{1}{n}\binom{n}{k+1}\binom{n}{k}$ counts the number of  Dyck paths of semilength $n$ with exactly $k$ valleys.  
\end{Proposition}

\begin{proof}

Let $k = {\rm dist}(e)$. Under the bijection between $WI_n$ and $\mathcal{D}_n$, the statistic ${\rm dist}(e)$ corresponds to the number of corners formed by the step sequence $EN$. Equivalently, it counts the number of peaks in the corresponding Dyck path. It is well known that, for Dyck paths,  the number of valleys is equal to the number of peaks minus 1. 

In the specialization $F_{n}(1;y,1,1,0)$, the variable $y$ records 
$$n-{\rm dist}(e)=n-k.$$ 
Since the Narayana numbers are symmetric, that is,
$$N(n,n-k-1)=N(n,k),$$  
we obtain the following identity
\begin{align*}
    F_{n}(1;y,1,1,0) &= \sum_{e\in WI_n} y^{\tel(e)}=\sum_{e\in WI_n} y^{n-\dist(e)}\\[5pt]
    &=\sum_{k=1}^{n} N(n,k-1) y^{n-k}=\sum_{k=0}^{n-1} N(n,k) y^{k}.
\end{align*}
\end{proof}

We now consider the statistic given by the number of zeros in $e\in WI_n$, which leads to a refinement in terms of Dyck paths of semilength $n$ with exactly $k$ returns, denoted by $T(n,k)$; see \cite{Deutsch-1999, OEIS}. These numbers satisfy
\[
T(n,k) = T(n-1,k) + T(n,k+1),
\]
with initial conditions $T(1,1)=1$ and $T(1,k)=0$ for $k\neq 1$.

 \begin{Proposition} \label{prop:yzp1q=0}
Let $q=0$ and $y=z=p=1$. For any $n\geq 1$, 
\[F_{n}(x;1,1,1,0) = \sum_{e\in WI_n} x^{\noz(e)}= \sum_{k=1}^{n} T(n,k)x^k.\]  
\end{Proposition}

 \begin{proof}

Let $t(n,k)$ denote the number of sequences $e\in WI_n$ containing exactly  $k$ zeros. Then
\[F_{n}(x;1,1,1,0) = \sum_{k=1}^{n} t(n,k)x^{k}. \]

We derive a recurrence for \( t(n,k) \) by constructing  sequences in \( WI_{n+1} \) from those in \( WI_n \). 

Fix $k$. Any sequence $e'\in WI_{n+1}$ with exactly $k$ zeros is obtained uniquely from a sequence $e\in WI_n$ with at least $k-1$ zeros, by increasing each entry by $1$  from the $k$-th position, and then prepending a $0$. This construction preserves the weakly increasing property and yields all such sequences. Therefore,
\[
t(n+1,k)=\sum_{i=k-1}^{n} t(n,i).
\]
Equivalently, this gives 
    \[
    t(n+1,k) = t(n,k-1) + t(n+1,k+1).
    \]
  With initial conditions $t(1,1)=1$ and $t(1,k)=0$ for $k\neq 1$, this recurrence coincides with that of the  number of Dyck paths of semilength $n$ containing $k$ returns  $T(n,k)$. Hence,
\[
t(n,k)=T(n,k),
\]
which completes the proof.
\end{proof}

\begin{Corollary}
For Dyck paths of semilength $n$, the height of the first peak and the number of returns  are equidistributed. More precisely, both statistics are counted by $T(n,k)$.
\end{Corollary}

\begin{proof}
By Proposition~\ref{prop:yzp1q=0}, the statistic $\noz(e)$ is distributed according to $T(n,k)$, 
which counts Dyck paths with $k$ returns. 

Under the natural correspondence between $WI_n$ and Dyck paths with semilength $n$, the statistic $\noz(e)$ corresponds to the height of the first peak. The result follows.
\end{proof}

\noindent \textbf{Remark}. This equidistribution also follows from the bijection constructed in 
\cite{Schmitt-Waterman-1994}, which implicitly relates the height of the first peak to the number of returns.

\subsection{Special Cases for $q=-1$}

Finally, we focus on the case $q=-1$. In this setting, the term $q^{\inv(e)}$ induces a sign-reversing effect. By constructing a sign-reversing involution on $I_n$, we show that the non-vanishing terms correspond precisely to the number of  self-inverse permutations in $\mathfrak{S}_n$, that is, permutations $\pi$ satisfying $\pi=\pi^{-1}$, or equivalently, permutations whose cycle decomposition consists only of cycles of length 1 and 2.

\begin{Proposition}
$F_n(1; 1, 1, 1, -1)$ equals the number of self-inverse permutations in $\mathfrak{S}_n$. 
\end{Proposition}

\begin{proof}
We define a sign-reversing involution $\tau: I_n \rightarrow I_n$ on the  set of inversion sequences.  The value of $F_n(1;1,1,1,-1)$ is determined by the fixed points of $\tau$, since all non-fixed points cancel in sign-reversing pairs.

The action of $\tau$ is determined by a backward scan of the sequence $e = (e_0, e_1, \ldots, e_{n-1})$ from $i=n-1$ down to $1$. Since $e_0=0$ in every inversion sequence, no operation is required at $0$. At each step, we inspect the entries according to the following recursive rules:
\begin{itemize}
    \item Current position: Let $i$ be the current index being inspected.
    \item Condition A: If $e_i= i$, then this position satisfies the fixed-point condition. We move to the previous index $i-1$ and repeat the process.

    \item Condition B: If $e_i\neq i$, we check the pair $(e_{i-1}, e_{i})$.
    \begin{itemize}
        \item[({B1})] If $e_i = e_{i-1}$, then this pair satisfies the fixed-point condition. We then return to Condition A and continue the  process from index $i-2$.
        \item[({B2})] If $e_i \neq e_{i-1}$, we define $\tau(e)$ by swapping the elements  $e_{i-1}$ and $e_{i}$. This transposition changes the inversion count by exactly one, effectively reversing the sign: 
        $$(-1)^{\inv(\tau(e))} = -(-1)^{\inv(e)}.$$   
    \end{itemize}
\end{itemize}

A sequence $e \in I_n$ is a fixed point of $\tau$ ($\tau(e) = e$) if and only if the entire sequence can be traversed from the end to the beginning by only applying {Condition A}  or {Condition (B1)}, never swapping a pair $(e_{i-1}, e_{i})$.
 
To count the number of fixed points, denoted $a(n)$, we establish a recurrence relation based on the construction rules:
\[
a(n) = a(n-1) + (n-1)a(n-2).
\]

We verify this recurrence by considering two mutually exclusive and exhaustive cases for the fixed point sequence $e = (e_0, e_1, \ldots, e_{n-1})$:
 
\noindent Case 1: $e_{n-1} = n-1$. The last element satisfies the fixed-point condition independently. The remaining prefix $(e_0, \ldots, e_{n-2})$ must be a fixed point of $\tau$ for $I_{n-1}$. There are $a(n-1)$ such sequences.

\noindent Case 2: $e_{n-1} = e_{n-2}$. The last two elements are tied as a pair. Since $e_{n-2}$ can take any value in $\{0, 1, \ldots, n-2\}$, there are $n-1$ possible choices for this pair. The remaining prefix $(e_0, \ldots, e_{n-3})$ must be a fixed point of $\tau$ for $I_{n-2}$. This yields $(n-1)a(n-2)$ sequences. 

Summing these two cases gives the recurrence relation. Since this recurrence is known to count the number of involutions in $\mathfrak{S}_n$, we conclude that $F_n(1; 1, 1, 1, -1)$ equals the number of self-inverse permutations in $\mathfrak{S}_n$. This completes the proof.
\end{proof}

For example, there are 10 inversion sequences for $n=4$ such that $\tau(e)=e$, that is, 
\[0000,0003,0011,0022,0023,0100,0123,0122,0111,0113.\]
The remaining $14$ inversion sequences satisfy that 
\[ \tau(0001) = 0010, ~\tau(0002) = 0020,~\tau(0101) = 0110,~ \tau(0102) = 0120,\]
\[ \tau(0021) = 0012, ~ \tau(0013) = 0103,~\tau(0121) = 0112.  \]

\section{Fixed frequencies of inversion sequences}
In the previous section, we explored various specializations of polynomial $F_n(x)$ by assigning specific values to its parameters. Those results provide a broad overview of how $F_n(x)$ unifies classical combinatorial sequences. 
In this section, we  focus on the fixed-frequency enumeration of inversion sequences.  For $e\in I_n$, let  $|e|_j$ denote the number of occurrences of $j$ in $e$, that is, 
$$|e|_j = \# \{i\in [0,n-1]: e_{i} = j\}.$$  
From the definition of inversion sequences, it is straightforward to verify that $|e|_j \leq n-j$. We now present a refined $q$-enumerator for sequences with a fixed frequency vector $(|e|_0, |e|_1, \dots, |e|_{n-1})$.

\begin{Theorem}\label{fix-inv}
Let $n \ge 1$. For any sequence of non-negative integers $(|e|_0, |e|_1, \dots, |e|_{n-1})$ such that $\sum_{j=0}^{n-1} |e|_j = n$,  we obtain
\begin{align}
\sum_{\substack{e\in I_n \\ |e|_j \text{ fixed for each $j$}}} q^{\inv(e)}
= \prod_{j=0}^{n-1} \qbinom{n-j - \sum_{s=j+1}^{n-1} |e|_s}{|e|_{j}}. \label{fixed-freq}
\end{align}\end{Theorem}

\begin{proof}
We construct the inversion sequence $e \in I_n$ by placing the entries in descending order of value, from $j = n-1$ down to $j=0$.  

Suppose we have already placed all entries with values  $\{j+1, j+2, \dots, n-1\}$. By the definition of inversion sequences,  the number of available positions  for the value $j$ is $m_j=n-j-\sum_{s=j+1}^{n-1} |e|_{s}$. 

We then choose $|e|_j$ positions out of these $m_j$ vacancies to insert the value $j$.  To calculate the inversions contributed by these entries, we treat the $|e|_j$ positions as  ``ones'' and the remaining $m_j - |e|_j$ empty slots (which will eventually contain values strictly smaller than $j$) as ``zeros''.  The generating function for these inversions is precisely the $q$-binomial coefficient   $\qbinom{m_j}{|e|_j}$.

  Iterating this process from $j=n-1$ down to $j=0$, the independence of each placement step ensures that the total generating function is the product of these coefficients. In particular, for $j=n-1$, the number of slots is $n-(n-1)=1$, yielding the initial factor $\qbinom{1}{|e|_{n-1}}$. Thus, the left-hand side of \eqref{fixed-freq} is equal to
 $$
    \qbinom{1}{|e|_{n-1}} \qbinom{2-|e|_{n-1} }{ |e|_{n-2}}   \cdots \qbinom{n-|e|_{n-1}-\cdots-|e|_1 }{ |e|_0}.
    $$
  This completes the proof.
  \end{proof}

\subsection{Augmented sequences and $q$-Stirling numbers}

Among the many possible fixed-frequency specializations, a particularly interesting case arises when all non-zero entries are pairwise distinct.  Specifically, it provides a direct combinatorial link to the $q$-Stirling numbers of the second kind, which have been extensively studied in the literature \cite{Carlitz-1933,  Gould-1961, Leroux-1990, Milne-1982, Sagan-1991}.
 Now, we define the $q$-Stirling numbers of the second kind by the recurrence:
\[S_q(n,k) = q^{n-k}S_{q}(n-1,k-1) + [k]_{q}S_q(n-1,k)\]
with the initial values $S_q(0,0)=1$ and $S_{q}(0,k)=0$ for any $k\neq 0$. 
When \( q = 1 \), this recurrence reduces to that of the classical Stirling numbers of the second kind, confirming that \( {S}_q(n,k) \) is a $q$-analog of the Stirling numbers of the second kind.

Next, let $I_{n,k}$ denote the subset of inversion sequences in $I_n$ containing exactly $k$ zeros, such that the remaining $n-k$ non-zero entries are pairwise distinct:
$$
I_{n,k} = \{ e \in I_n : |e|_0 = k \text{ and } |\{e_i:   e_{i}>0\} | = n-k \}. 
$$
For any $e \in I_{n,k}$, let $S = \{e_i : e_i > 0\}$ be the set of its non-zero entries, and let $R = \{1, 2, \dots, n-1\} \setminus S$ denote the set of excluded values, which has cardinality $|R| = k-1$. We represent this structure by an \textit{augmented sequence} $e \cdot R$, where the elements of $R$ are appended to $e$ in increasing order, separated by a dot.
  
   For example, if $n=8$ and $k=5$, an inversion sequence $e \in I_{8,5}$ has $n-k=3$ distinct non-zero entries. If $S = \{1, 3, 4\}$, then $R = \{2, 5, 6, 7\}$ and a possible augmented sequence would be: 
   $01003400 \cdot 2567$.
   If $k=1$, the set of excluded values $R$ is empty. If $k=n$, then $R = \{1, 2, \dots, n-1\}$, meaning all non-zero values are omitted from the sequence $e$.

We define a new statistic ${\rm Inv}(e)$ on the augmented   sequence by
\[
{\rm Inv}(e) =\inv(e\cdot R) =  \inv(e) + \# \{ (e_i, r_j) : e_i > r_j,~e_i \in e,~r_j \in R \}.
\]
With this definition, we can now connect these augmented sequences to the $q$-Stirling numbers of the second kind.

\begin{Theorem}
    The $q$-Stirling numbers of the second kind  $S_q(n,k)$ satisfy that 
    \[S_q(n,k) =\sum_{e\in I_{n,k}} q^{ {\rm Inv} (e)}. \]
\end{Theorem}

 \begin{proof}
To obtain an augmented sequence in $I_{n,k}$ from an element of $I_{n-1}$, we examine how to position the $n$-th entry $e_{n-1}$.  
 
 \noindent Case 1: Let  $e' \in I_{n-1, k-1}$.  
 Setting $e_{n-1}=0$ increases the total number of zeros to $k$.  It creates  inversions with all $n-k$ non-zero elements already in the sequence. This contributes a factor of $q^{n-k}$, corresponding to the term $q^{n-k}S_{q}(n-1,k-1)$ in the recurrence.  
 
 \noindent Case 2: Let $e' \in I_{n-1, k}$.  To maintain $k$ zeros, the new entry $e_{n-1}$ must be a non-zero value selected from the available set $R' \cup \{n-1\}$. There are exactly $k$ such choices. If we choose the $i$-th smallest available value, it contributes $i-1$ additional inversions in the augmented sequence. Summing over all possibilities  yields the $q$-weight  $[k]_q$, corresponding to the term $[k]_qS_{q}(n-1,k)$.
 
 The base case $n=1$ is trivial as $e=(0)$ has ${\rm Inv}(e)=0$, giving $S_q(1,1)=S_q(0,0)=1$. The proof is completed.
\end{proof}

\section{$q$-derivative operator and inversion sequences}

A primary motivation for this work arises from the study of the $q$-analog of the Comtet identity. Recall that Comtet derived an expansion formula for the iterated derivative operator $(gD)^n$; see \eqref{Ank-Comtet}. In this work, we investigate the operator $(gD_q)^n$, obtained by replacing the classical derivative $D$ with the $q$-derivative operator $D_q$. We demonstrate that the expansions generated by $(gD_q)^n$ admit a canonical and explicit combinatorial interpretation in terms of inversion statistics on inversion sequences $I_n$.

The following $q$-Leibniz formula is a useful tool for computing higher-order $q$-derivatives, see \cite[(4.2)]{Foata-Han-2012}.
\begin{Theorem}[$q$-Leibniz]\label{th:qLeibniz}
The $q$-derivative of a product of functions satisfies:
$$
D_q \prod_{1 \le i \le n} p_i(x)
= \sum_{1 \le i \le n}
 p_1(x) \cdots p_{i-1}(x) \cdot (D_q p_i(x)) \cdot
 p_{i+1}(xq) \cdots p_n(xq).
$$
\end{Theorem}

\subsection{The iterated $q$-derivative operator}
First, we introduce the scaled  $q$-derivative of a function $f(x)$.  Let $ f_i(x)  =(D_q^i f)(x)$, and  define the scaled version by:  
\begin{equation}
    f_i^{(j)}(x) :=f_i(xq^j)= (D_q^i f)(q^j x).
\end{equation}
From the definition of the $q$-derivative, we obtain the following scaling relation: 
 \begin{equation}\label{eq:scaling}
     D_q f_i^{(j)}(x)  = q^j f_{i+1}^{(j)}(x).
 \end{equation}
With this notation, we derive the  first few iterations of 
$ (g D_q)^nf$: 
\begin{align*}
(g D_q)^1 f &=  \tg00 \tf10,\\[5pt]
(g D_q)^2 f  &=\tg00  \tg10 \tf11 + \tg00 \tg00 \tf20,\\[5pt]
(g D_q)^3 f 
&=( \tg00 \tg10 \tg11 +  \tg00 \tg00 \tg20 ) \tf12\\[5pt]
& \quad + ((1+q)\tg00 \tg00\tg10 +  \tg00  \tg10 \tg01 ) \tf21\\[5pt]
& \quad + (\tg00 \tg00\tg00)\tf30.
\end{align*}

By extending these low-order instances to higher-order cases via the $q$-Leibniz rule, we prove that the iterated operator $(gD_q)^n f$ admits a natural combinatorial representation in terms of inversion statistics on inversion sequences, as stated  in the following theorem.
\begin{Theorem}\label{gD-inv} 
For each $n\geq 1$, we have
\begin{equation}
(gD_q)^n f =  \sum_{e \in I_n } q^{\inv(e)} \cdot \tg{0}{0}\tg{k_1}{0}  \tg{k_2}{K_1}  \cdots \tg{k_{n-1}}{K_{n-2}}  f_k^{(K_{n-1})},
\end{equation}
where $k=\noz(e)$, $k_j=|e|_{n-j}$, and $K_j = k_1+k_2+\cdots+k_{j}$.
\end{Theorem}

\begin{proof}
We proceed by induction on $n$. When $n=1$, the formula reduces to
\[(gD_q)f = \sum_{e\in I_1}q^0\tg{0}{0} f_1^{(0)} = \tg{0}{0} f_1^{(0)}.\]
 The base case holds. 

Assume that the formula holds for $n$. We prove it for $n+1$. Applying the $q$-Leibniz rule to $(gD_q)^nf$, we obtain 
\begin{align*}
(gD_q)^{n+1} f &= 
\sum_{e \in I_n } W_1(e)
+ \sum_{e \in I_n }\sum_{i=0}^{n-2} W_2(e,i) 
+ \sum_{e \in I_n } W_3(e),
\end{align*}
where
\begin{align*}
W_1(e) &= 
q^{\inv(e)}\cdot  (D_q(\tg{0}{0}))  
\tg{k_1}{1}  
\tg{k_2}{K_1+1}  
\cdots \tg{k_{n-1}}{K_{n-2}+1}  f_k^{(K_{n-1}+1)},\\[5pt]
W_2(e,i)&= 
q^{\inv(e)}\cdot  \tg{0}{0}  \tg{k_1}{0}   \cdots (D_q (g_{k_{i+1}}^{(K_i)}))  
\cdots \tg{k_{n-1}}{K_{n-2}+1}  f_k^{(K_{n-1}+1)},\\[5pt]
W_3(e)&=
 q^{\inv(e)}\cdot \tg{0}{0}\tg{k_1}{0}  \tg{k_2}{K_1}  \cdots \tg{k_{n-1}}{K_{n-2}}  (D_q(f_k^{(K_{n-1})})).
\end{align*}
For $e'\in I_{n+1}$, let
\begin{equation*}
W(e') =  q^{\inv(e')} \cdot \tg{0}{0}\tg{k_1'}{0}  \tg{k_2'}{K_1'}  \cdots \tg{k_{n}'}{K_{n-1}'}  f_k^{(K_{n}')},
\end{equation*}
where $k'=\noz(e')$, $k_j'=|e'|_{n+1-j}$, and $K_j' = k_1'+k_2'+\cdots+k_{j}'$.

Let $e = (e_0, e_1, \ldots, e_{n-1}) \in I_n$ be a given inversion sequence. By the definition of the operator $D_q$, together with the fundamental properties of inversion sequences, we verify that:

(1) Let $e' = (e_0, e_1, \ldots, e_{n-1} , n) \in I_{n+1}$, then 
$W_1(e)=W(e')$.

(2) For each $0\leq i \leq n-2$, let $e' = (e_0, e_1, \ldots, e_{n-1} , n-i-1) \in I_{n+1}$, then 
$W_2(e,i)=W(e')$.

(3) Let $e' = (e_0, e_1, \ldots, e_{n-1} , 0) \in I_{n+1}$, then 
$W_3(e)=W(e')$.

These three cases exhaust all possible extensions of $e$ to an inversion sequence in $I_{n+1}$, and the corresponding weights match exactly.  Hence,
$$
(gD_q)^{n+1} f = 
\sum_{e' \in I_{n+1} } W(e'),
$$
which completes the induction.
\end{proof}

We now turn to the $q$-analog of  $L_{n,k}$ in Comtet's expansion formula~\eqref{Ank-Comtet}.    Let the iterated operator be expressed as: 
\begin{equation}\label{def:Lnkq}
(gD_q)^n f = \sum_{k=1}^n L_{n,k}(q) \tf{k}{n-k},  \qquad (n\geq 1)
\end{equation}
where each $L_{n,k}(q)$   depends only on $g$ and is independent of $f$. The following theorem provides an explicit $q$-analog of Comtet's formula for $L_{n,k}(q)$.
\begin{Theorem}\label{th1:q-Comtet}
The explicit form of $L_{n,k}(q)$ is given by
\begin{equation}\label{eq:qComtet}
\frac{\tg{0}{0}}{[k]_q!} \sum [2-K_1]_q[3-K_2]_q \cdots [n-K_{n-1}]_q \frac{\tg{k_1}{0}}{[k_1]_q!}\frac{\tg{k_2}{K_1}}{[k_2]_q!} \cdots\frac{\tg{k_{n-1}}{K_{n-2}}}{[k_{n-1}]_q!},
\end{equation}
where the summation is over all sequences $(k_1,k_2,\ldots,k_{n-1})$ of nonnegative integers such that
$k_1+k_2+\cdots+k_{n-1}=n-k$ and $K_j=k_1+\cdots+k_j\leqslant j$ for any $1\leqslant j\leqslant n-1$.
\end{Theorem}
  
\begin{proof}
Fix $1 \leq k \leq n$, and let $(k_1, k_2, \ldots, k_{n-1})$ be a sequence as specified in Theorem \ref{th1:q-Comtet}.
By \eqref{def:Lnkq} and Theorem \ref{gD-inv}, we obtain
\begin{equation}\label{eq:L:inve}
L_{n,k}(q) = \sum_{e \in I_n} q^{\inv(e)} \cdot \tg{0}{0}\tg{k_1}{0}\tg{k_2}{K_1} \cdots \tg{k_{n-1}}{K_{n-2}},
\end{equation}
where the summation ranges over all inversion sequences $e \in I_n$ such that $\noz(e)=k$ and $k_j = |e|_{n-j}$ for each $j$. Hence, this summation can be evaluated by applying Theorem \ref{fix-inv}. 
Substituting $|e|_j$ by $k_{n-j}$ in the right-hand side of \eqref{fixed-freq}, we obtain
\begin{align*}
\prod_{j=0}^{n-1} \qbinom{n-j - \sum_{s=j+1}^{n-1} |e|_s}{|e|_{j}}
&= \prod_{j=1}^{n} \qbinom{j - K_{j-1}}{k_{j}} \\
&= \frac{[2-K_1]_q[3-K_2]_q \cdots [n-K_{n-1}]_q}{[k]_q![k_1]_q!\cdots [k_{n-1}]_q!}.
\end{align*}
This identity implies \eqref{eq:qComtet}, which completes the proof.
\end{proof}

While Theorem \ref{th1:q-Comtet} gives a closed form for $L_{n,k}(q)$, a recurrence relation is often also convenient for applications, particularly when $g(x)$ and $f(x)$ are chosen as specific functions. The following proposition establishes such a recurrence.

\begin{Proposition} For $n\geq 0$, the coefficient $L_{n,k}(q)$ satisfies the following recurrence relations: 
\begin{align*}
L_{0,0}(q)&=1,\\[5pt]
L_{n+1, 1}(q) &=g D_qL_{n,1}(q),\\[5pt]
L_{n+1, n+1}(q) &=g L_{n,n}(q),\\[5pt]
L_{n+1, k}(q) &=g D_qL_{n,k}(q) + gq^{n-k+1} L_{n,k-1}(q), \text{\ for\ } 1<k\leq n. 
\end{align*}
\end{Proposition}
\begin{proof}

We proceed by applying the operator $gD_q$ to the expansion
\[
(gD_q)^n f = \sum_{k=1}^n L_{n,k}(q) \,\tf{k}{n-k}.
\]
Using the $q$-Leibniz rule, we obtain
\begin{align*}
&gD_q \Bigl( \sum_{k=1}^n L_{n,k}(q) \,\tf{k}{n-k} \Bigr)\\[5pt]
&= g \sum_{k=1}^n D_q L_{n,k}(q) \,\tf{k}{n-k+1} 
+ g \sum_{k=1}^n L_{n,k}(q) D_q \tf{k}{n-k}\\[5pt]
&= g \sum_{k=1}^n D_q L_{n,k}(q) \,\tf{k}{n-k+1} 
+ g \sum_{k=1}^n L_{n,k}(q)q^{n-k}\tf{k+1}{n-k}\\[5pt]
&= g \sum_{k=1}^n D_q L_{n,k}(q) \,\tf{k}{n-k+1} 
+ g \sum_{k=2}^{n+1} L_{n,k-1}(q)q^{n-k+1}\tf{k}{n-k+1}.
\end{align*}
By definition, we have
\[
(gD_q)^{n+1} f = \sum_{k=1}^{n+1} L_{n+1,k}(q) \,\tf{k}{n+1-k}.
\]
By comparing the coefficients of $\tf{k}{n+1-k}$ in the preceding expansion, 
we obtain the desired recurrence relations for $L_{n,k}(q)$.
\end{proof}

\subsection{Applications for special choices of $g(x)$ and $f(x)$} 

In this subsection, we obtain $q$-Scherk identity and illustrate how it leads naturally to several classical $q$-identities under suitable specializations of $f(x)$.

Let \( g(x) = x \). For \( j \geq 0 \), we have 
$$
g_{i}^{(j)}=
\begin{cases}
    q^jx, & i=0,\\[5pt]
    1, & i=1,\\[5pt]
    0, &i\geq 2.
\end{cases}
$$
Substituting $g_i^{(j)}$ into the recurrence of  $L_{n,k}(q)$  and using $D_qx^k = [k]_qx^{k-1}$, we obtain
$L_{n,k}(q) = x^k \hat{S}_q(n,k)$,
where $\hat{S}_q(n,k)$ satisfies 
the same recurrence and initial values as $S_q(n,k)$ defined in Section 4.  
Consequently, we obtain the following $q$-Scherk identity: 
\begin{Theorem}[$q$-Scherk]\label{th:q-Scherk}
For $n\geq 1$, 
    \begin{equation}\label{eq:q-Scherk}
        (xD_q)^nf=\sum_{k=1}^{n} S_{q}(n,k)x^kf_k^{(n-k)}.
    \end{equation}
\end{Theorem}

Consider the function \( f(x) = \frac{1}{1 - x} \). A direct computation shows 
$$
(xD_q)^n \frac{1}{1-x} 
= \sum_{\ell \geq 0} [\ell]_q^n x^\ell 
	= x\sum_{\ell \geq 1} [\ell]_q^n x^{\ell -1}.
$$
On the other hand,  the Euler-Mahonian polynomial 
$$
A_n(x,q) = \sum_{\sigma \in \mathfrak{S}_n} x^{\des(\sigma)} q^{\maj(\sigma)}
$$	
satisfies the identity of  Carlitz, namely, 
\begin{equation}\label{eq:Carlitz}
\frac{A_n(x,q)}{(x;q)_{n+1}} = \sum_{\ell \ge 0} x^\ell \big([\ell+1]_q\big)^n.
\end{equation}
Comparing these expressions yields
\begin{Theorem}\label{th:EuMa}
For $n\geq 0$, 
$$
(xD_q)^n \frac{1}{1-x} = \frac{xA_n(x,q)}{(x;q)_{n+1}}  .
$$
\end{Theorem}

Applying the $q$-Scherk identity \eqref{eq:q-Scherk} and using 
$$
D_q^jf(x) = \frac{[j]_q!}{(x;q)_{j+1}},
$$
we have 
\begin{align*}
(xD_q)^n \frac{1}{1-x}
&=\frac{\sum_{j=1}^{n} S_q(n,j) x^{j}(x;q)_{n-j} [j]_q!}{(x;q)_{n+1}}.
\end{align*}
Comparing with Theorem \ref{th:EuMa}, we arrive at the identity.
\begin{Theorem}
For $n\geq 1$, we have 
	$$\sum_{j=1}^{n} S_q(n,j) x^{j}(x;q)_{n-j} [j]_q! = x A_n(x,q).$$
\end{Theorem}

\noindent \textbf{Remark.} Garsia \cite[Equation (I.16)]{Garsia-1979} proved algebraically that 
$$
\sum_{j=1}^{n} \widetilde{S}_q(n,j) x^{j}(xq^{j+1};q)_{n-j} [j]_q! = x A_n(x,q),
$$
where $\widetilde{S}_q(n,j)$ is the $q$-Stirling numbers of the second kind satisfying 
$$
\widetilde{S}_q(n+1,j) =[j]\widetilde{S}_q(n,j)+q^{j-1}\widetilde{S}_q(n,j-1) 
$$
with the initial value $\widetilde{S}_q(1,1)=1$, defined by Milne \cite{Milne-1982}. Here, we note that 
\[S_q(n,j) = q^{(j-1)(2n-j)/2}\widetilde{S}_{1/q}(n,j).\]

Continuing to apply the $q$-Scherk identity to the monomial $x^{k}$, we obtain  
     \begin{align*}
    (xD_q)^n x^k 
    	&=\sum_{j=1}^{n} S_q(n,j) x^{j}[k]_q[k-1]_q \cdots [k-j+1]_q\ 
    	(q^{n-j}x)^{k-j}.
    \end{align*}
    Summing over $k\ge 0$ and comparing coefficients of $x^k$ yields the following identity.
\begin{Theorem} For $n\geq 1$,
$$
	[k]_q^n =
	\sum_{j=1}^{n} S_q(n,j)   [k]_q[k-1]_q \cdots [k-j+1]_q\ q^{(n-j)(k-j)}.
$$
\end{Theorem}

\noindent \textbf{Remark.} Leroux and M\'edicis (\cite[Theorem 4.3]{Medicis-Leroux-1993}) obtained a related expansion 
 $$
	[k]_q^n =
	\sum_{j=1}^{n} S_q^{*}(n,j)   [k]_q[k-1]_q \cdots [k-j+1]_q\  q^{\frac{j(j-1)}{2}},
$$
 where the $q$-Stirling number of the second kind $S_q^{*}(n,k)$ satisfies 
\begin{equation}
    S_q^{*}(n+1,k)=[k]_qS_{q}^{*}(n,k)+S_q^{*}(n,k-1).
\end{equation}
These two families are related by \( S_q(n,k) = q^{(k-1)(n-k)}S_{1/q}^{*} (n,k) \).

\bibliographystyle{plain}
%\bibliography{invseq}
%\bibliography{invseq.bib}

\end{document}